\documentclass[reqno]{amsart}
\usepackage{hyperref}

\begin{document}
\title[
Asymptotic Behavior of neural fields ] {Asymptotic Behavior of
neural fields in an unbounded domain}

\author[S. H. da Silva 
]
{Severino Hor\'acio da Silva}
\author[M. B. Silva 
]
{Michel Barros Silva}

\address{Severino Hor\'acio da Silva$^1$ \newline
Unidade Acad\^emica de Matem\'atica e Estat\'istica UAME/CCT/UFCG\\
 Rua Apr\'igio Veloso, 882,  Bairro Universit\'ario CEP 58429-900,
\ Campina Grande-PB, Brasil} \email{horacio@dme.ufcg.edu.br}

\address{Michel Barros Silva$^2$ \newline
Unidade Acad\^emica de Matem\'atica e Estat\'istica UAME/CCT/UFCG\\
 Rua Apr\'igio Veloso, 882,  Bairro Universit\'ario CEP 58429-900,
 Campina Grande-PB, Brasil}
\email{michel@dme.ufcg.edu.br}
\thanks{$^1$Supported by  CAPES/CNPq-Brazil}
 \thanks{$^2$Supported by  CNPq-Brazil}


\begin{abstract}
 In this paper, we prove the existence of a compact global attractor for the flow generated by equation
 $$
 \frac{\partial u}{\partial t}(x,t)+u(x,t)= \int_{\mathbb{R}^{N}}J(x-y)(f( u(y,t))dy+
 h,   \quad h  > 0, \quad x\in \mathbb{R}^{N}, \quad t\in\mathbb{R}_{+}
 $$
 in the weight space  $L^{p}(\mathbb{R}^{N}, \rho)$. We also give uniform estimates on the size of the attractor and we exhibit a  Lyapunov functional to the flow generated by this equation.
\end{abstract}

\maketitle
\numberwithin{equation}{section}
\newtheorem{theorem}{Theorem}[section]
\newtheorem{lemma}[theorem]{Lemma}
\newtheorem{proposition}[theorem]{Proposition}
\newtheorem{remark}[theorem]{Remark}

\section{Introduction}

In this work we consider the non local evolution equation
\begin{equation}
\frac{\partial u}{\partial t}(x,t)=-u(x,t)+ J*(f\circ u)(x,t)+ h,
\quad h
> 0, \label{1.1}
\end{equation}
where $u(x,t)$ is a real-valued function on $\mathbb{R}^{N}\times
\mathbb{R}_{+}$, $h$ is a positive constant,
$J \in C^{1}(\mathbb{R}^{N})$ is a non negative even function supported in the
ball of center at the origin and radius $1$, and, $f$ is a non negative nondecreasing
function. The $*$ above denotes convolution product in $\mathbb{R}^{N}$, namely:
\begin{equation}
(J*u)(x)=\int_{\mathbb{R}^{N}}J(x-y)u(y)dy \label{1.2}.
\end{equation}

The function $u(x,t)$ denotes the mean membrane potential of a patch
of tissue located at position $x\in \mathbb{R}^{N}$ at time $t\geq
0$. The connection function $J(x)$ determines the coupling between
the elements at position $x$ and position $y$. The non negative
nondecreasing function $f(u)$ gives the neural firing rate, or
averages rate at which spikes are generated, corresponding to an
activity level $u$. The neurons at a point $x$ are said to be active
if $f(u(x,t))>0$. The parameter $h$ denotes a constant external
stimulus applied uniformly to the entire neural field.

For the particular case, where $N=1$, there are already in the
literature several works dedicated to the analysis of this model,
(see, for example, \cite{Amari}, \cite{Chen}, \cite{Ermentrout},
\cite{Kishimoto}, \cite{Krisner}, \cite{Laing}, \cite{Rubin},
\cite{Silva}, \cite{Silva2}, \cite{Silva3}, \cite{Silva4},
\cite{Silva5}, \cite{Wilson} and \cite{Wu}). Also there are some
works for this model with $N> 1$, (see for example \cite{French} and
\cite{Laing}).

In this paper we extend, for $L^{p}(\mathbb{R}^{N}, \rho)$, $N\geq
1$ and $1<p<\infty$, the results (on global attractors) obtained in
\cite{Silva2} in the phase space $L^{p}(\mathbb{R}, \rho)$.
Furthermore, we exhibit a Lyapunov functional to the flow generated
by (\ref{1.1}).

This paper is organized as follows. In Section 2 we prove that, in
the phase space $L^{p}(\mathbb{R}^{N},\rho)=\{u\in L^{1}_{\rm
loc}(\mathbb{R}^{N}) : \int |u|^{p}\rho(x)dx<\infty\}$, the Cauchy
problem for \eqref{1.1} is well posed with globally defined
solutions. In Section 3 we prove that the system is dissipative in
the sense of \cite{Hale},  that is, it has a global compact
attractor, generalizing Theorem 3.5 of \cite{Silva2}. In our proof,
we use the Sobolev's compact embedding
$W^{1,p}(B[0,l])\hookrightarrow L^{p}(B[0,l])$ and the same
techniques used in \cite{Pereira} and \cite{Silva2} (see also
\cite{Bezerra}, \cite{Severino} and \cite{Severino2} for related
work). In Section 4, we prove some estimates for the attractor and
finally, in Section 5, using ideas from \cite{Giese}, \cite{Kubota}
and \cite{Silva4}, we exhibit a Lyapunov function for the flow
genereted by \eqref{1.1}.

\section{Well-posedness}

In this section we consider the flow generated by \eqref{1.1} in the
space $L^{p}(\mathbb{R}^{N},\rho)$ defined by
\begin{align*}
L^{p}(\mathbb{R}^{N},\rho)=\big\{u\in L^{1}_{\rm loc}(\mathbb{R}^{N}) :
\int_{\mathbb{R}^{N}}|u(x)|^{p}\rho(x)dx <+\infty\big\},
\end{align*}
with norm
$\|u\|_{L^{p}(\mathbb{R}^{N},\rho)}=\left(\int_{\mathbb{R}^{N}}|u(x)|^{p}\rho(x)dx
\right)^{1/p}$. Note that, in this space, the constant function equal to 1 has norm
1.

As similarly assumed in \cite{Silva2}, we assume here the
following hypotheses on the functions $f$ and $\rho$:

\begin{itemize}
\item[(H1)] the function $f:\mathbb{R}\to \mathbb{R}$ is globally Lipschitz,
that is, there exists $k_{1}>0$ such that
\begin{equation}
|f(x)-f(y)|\leq k_{1}|x-y|, \quad  \forall \, x,y \in
\mathbb{R},\label{1.4}
\end{equation}
\item[(H2)] $\rho: \mathbb{R}^{N} \rightarrow \mathbb{R}$ is an integrable positive even function
 with $\int_{\mathbb{R}^{N}}\rho(x)dx=1$ and there exists constant $K>0$ such that
$$
\sup\{\rho(x) :x\in \mathbb{R}^{N},\; |x-y | \leq1\}\leq K\rho(y),
\quad \forall \,y\in \mathbb{R}^{N}.
$$

\end{itemize}

The corresponding higher-order Sobolev space
$W^{k,p}(\mathbb{R}^{N},\rho)$ is the space of functions $u\in
L^{p}(\mathbb{R}^{N},\rho)$ whose distributional derivatives up to
order $k$ are also in $L^{p}(\mathbb{R}^{N}, \rho)$, with norm
$$
\|u\|_{W^{k,p}(\mathbb{R}^{N},\rho)}=\Big(\sum_{i=1}^{k}\left\|\frac{\partial^{i}
u}{\partial
x^{i}}\right\|_{L^{p}(\mathbb{R}^{N},\rho)}^{p}\Big)^{1/p}.
$$

\begin{lemma} \label{lem2.1}
Suppose that {\rm  (H2)} holds. Then
$$
\|J*u\|_{L^{p}(\mathbb{R}^{N},\rho)}\leq K^{1/p}
\|J\|_{L^{1}}\|u\|_{L^{p}(\mathbb{R}^{N},\rho)}.
$$
\end{lemma}

\begin{proof}
 Since $J$ is bounded and compact supported, $(J*u)(x)$ is well
defined for $u\in L^{1}_{\rm loc}(\mathbb{R}^{N})$. Thus,
\begin{align*}
\|J*u\|_{L^{p}(\mathbb{R}^{N},\rho)}^{p}&= \int_{\mathbb{R}^{N}}|(J*u)(x)|^{p}\rho(x)dx
\\
&= \int_{\mathbb{R}^{N}}\left| \int_{\mathbb{R}^{N}} J(x-y)u(y)dy\right|^p\rho(x)dx
\\
&\leq \int_{\mathbb{R}^{N}}\left(\int_{\mathbb{R}^{N}} |J(x-y)||u(y)|dy\right)^p\rho(x)dx.
\end{align*}
Using
\begin{equation}
\frac{p-1}{p} + \frac{1}{p} = 1,  \label{conjugado}
\end{equation}
we obtain
\begin{equation*}
||J*u\|_{L^{p}(\mathbb{R}^{N},\rho)}^{p} \leq \int_{\mathbb{R}^{N}}\left(\int_{\mathbb{R}^{N}}|J(x-y)|^{(p-1)/p}|J(x-y)|^{1/p}|u(y)|dy\right)^p\rho(x)dx.
\end{equation*}
By Holder's inequality (see \cite{Brezis}), we have
\begin{eqnarray*}
\lefteqn{\|J*u\|_{L^{p}(\mathbb{R}^{N},\rho)}^{p}\leq}
\\
&\leq& \int_{\mathbb{R}^{N}}\left(\left(\int_{\mathbb{R}^{N}}|J(x-y)|dy\right)^{(p-1)/p}\left(\int_{\mathbb{R}^{N}}|J(x-y)||u(y)|^pdy\right)^{1/p}\right)^p\rho(x)dx
\\
&=& \int_{\mathbb{R}^{N}}||J||_{L^1}^{p-1}\left(\int_{\mathbb{R}^{N}}|J(x-y)||u(y)|^pdy\right)\rho(x)dx
\\
&=& ||J||_{L^1}^{p-1}\int_{\mathbb{R}^{N}}\int_{\mathbb{R}^{N}}|J(x-y)||u(y)|^pdy\ \rho(x)dx.
\end{eqnarray*}
Denoting the closed ball of center $ y $ and radius $1$ by $B[y,1]$
and using \eqref{1.2} and (H2), follows that
\begin{align*}
\|J*u\|_{L^{p}(\mathbb{R}^{N},\rho)}^{p}&\leq
\|J\|_{L^{1}}^{p-1}\int_{\mathbb{R}^{N}}\Big(\int_{B[y,1]}J(x)\rho(x)dx\Big)|u(y)|^{p}dy
\\
&\leq \|J\|_{L^{1}}^{p-1}\int_{\mathbb{R}^{N}}\Big(K\rho(y)\int_{B[y,1]}J(x)dx\Big)|u(y)|^{p}dy
\\
&\leq K\|J\|_{L^{1}}^{p}\int_{\mathbb{R}^{N}}|u(y)|^{p}\rho(y)dy
\\
&= K\|J\|_{L^{1}}^{p}\|u\|^{p}_{L^{p}(\mathbb{R}^{N},\rho)}.
\end{align*}
It conclude the result.
\end{proof}

\begin{proposition} \label{prop5.2}
Suppose that the hypotheses {\rm (H1)} and {\rm (H2)} hold. Then the function
$$
F(u)=-u+J*(f \circ u) +h
$$
is globally Lipschitz in $L^{p}(\mathbb{R}^{N}, \rho)$.
\end{proposition}

\begin{proof}
  From triangle inequality and Lemma \ref{lem2.1}, it follows that
\begin{align*}
\|F(u)-F(v)\|_{L^{p}(\mathbb{R}^{N}, \rho)}&\leq \|v-u\|_{L^{p}(\mathbb{R}^{N}, \rho)}+\|J*(f\circ u)-J*(f\circ v)\|_{L^{p}(\mathbb{R}^{N}, \rho)}
\\
&\leq  \|v-u\|_{L^{p}(\mathbb{R}^{N},\rho)}+K^{1/p}\|J\|_{L^{1}}\|(f\circ u)-(f\circ v)\|_{L^{p}(\mathbb{R}^{N}, \rho)}.
\end{align*}
We have
\[
\|(f\circ u)-(f\circ v)\|_{L^{p}(\mathbb{R}^{N},\rho)}^{p}
\leq  \int_{\mathbb{R}^{N}}k_{1}^{p}|u(x)-v(x)|^{p}\rho(x)dx
= k_{1}^{p}\|u-v\|_{L^{p}(\mathbb{R}^{N}, \rho)}^{p}.
\]
Then
\[
\|F(u)-F(v)\|_{L^{p}(\mathbb{R}^{N},\rho)}\leq
(1+K^{1/p}\|J\|_{L^{1}}k_{1})\|u-v\|_{L^{p}(\mathbb{R}^{N}, \rho)}.
\]
Therefore, $F$ is globally Lipschitz in $L^{p}(\mathbb{R}^{N}, \rho)$.
\end{proof}

\begin{remark} \label{rmk2.5} \rm
From Proposition \ref{prop5.2} and standard results of ODEs in
Banach spaces (see  \cite{Daleckii}), follows that the Cauchy
problem for \eqref{1.1} is well posed in $L^{p}(\mathbb{R}^{N},
\rho)$ with globally defined solutions.
\end{remark}

\section{Existence of a global attractor}

In this section, we prove the existence of a global maximal
invariant compact set $\mathcal{A}\subset L^{p}(\mathbb{R}^{N},
\rho)$ for the flow of \eqref{1.1}, which attracts each bounded set
of $L^{p}(\mathbb{R}^{N}, \rho)$ (the global attractor, see
\cite{Hale} and \cite{Teman}), generalizing Theorem 3.3 in
\cite{Silva2}. For this, beyond (H1) and (H2) we assume the
following additional hypotheses:

\begin{itemize}

\item[(H3)] there exists $a>0$ such that $|f(x)|\leq a, \,\, \forall \,\, x\in \mathbb{R}$;\\

\item[(H4)] the non negative, symmetric bounded
function $J$ has bounded derivative with
\begin{eqnarray*}\label{norma-derivada-J}
\sup_{x\in \mathbb{R}^{N}}\int_{\mathbb{R}^{N}}\partial_{
x}J(x-y)dy\leq S\ \ \mbox{and}\ \ \sup_{y\in
\mathbb{R}^{N}}\int_{\mathbb{R}^{N}}\partial_{ x}J(x-y)dx\leq S,
\end{eqnarray*}
for some constant $0<S<\infty.$

\end{itemize}

From now on 
we 
denote by $S(t)$ the flow generated by \eqref{1.1}.\\

We recall that a set $\mathcal{B} \subset L^{p}(\mathbb{R}^{N}, \rho)$ is
an absorbing set for the flow $S(t)$ in $L^{p}(\mathbb{R}^{N}, \rho)$
if, for any bounded set $B \subset L^{p}(\mathbb{R}^{N}, \rho)$, there
is a $t_{1}>0$ such that $S(t)B \subset \mathcal{B}$ for any $t\geq
t_{1}$, (see \cite{Teman}).


\begin{lemma}
Suppose that the hypotheses (H1), (H2) and (H3) hold and let $R=aK^{1/p}\|J\|_{L^{1}}+h$. Then  the ball with center at the origin   and radius $R+\varepsilon$ is an absorbing set for the flow $S(t)$ in  $L^{p}(\mathbb{R}^N, \rho)$  for any $\varepsilon >0$.\label{lema3.1}
\end{lemma}
\begin{proof} Let $u(x,t)$ be the solution of (\ref{1.1}) with initial condition $u(\cdot,0)\in L^{p}(\mathbb{R}^{N},\rho)$. Then,  by the variation of constants formula,
$$
u(x,t)=e^{-t}u(x,0)+\int_{0}^{t}e^{s-t}[J*(f\circ u)(x,s)+h]ds.
$$
Hence

\begin{eqnarray*}
\|u(\cdot,t)\|_{L^{p}(\mathbb{R}^N, \rho)}&\leq& ||e^{-t}u(\cdot,0)||_{L^p(\mathbb{R}^N,\rho)}
\\
&+& \int_0^te^{s-t}\left|\left|J\ast(f\circ u)(\cdot,s)+h\right|\right|_{L^p(\mathbb{R}^N,\rho)}ds
\\
&\leq&e^{-t}||u(\cdot,0)||_{L^p(\mathbb{R}^N,\rho)}
\\
&+& \int_0^te^{s-t}[||J\ast(f\circ u)(\cdot,s)||_{L^p(\mathbb{R}^N,\rho)}+h]ds.
\end{eqnarray*}
Then, using Lemma \ref{lem2.1}, it follows that
\begin{eqnarray*}
\|u(\cdot,t)\|_{L^{p}(\mathbb{R}^N, \rho)} &\leq&  e^{-t}||u(\cdot,0)||_{L^p(\mathbb{R}^N,\rho)}
\\
&+& \int_0^te^{s-t}[K^{1/p}||J||_{L^1}||f(u(\cdot,s))||_{L^p(\mathbb{R}^N,\rho)}+h]ds
\end{eqnarray*}
Now, from (H3), we have
\begin{eqnarray*}
\|f(u(\cdot,s))\|_{L^{p}(\mathbb{R}^{N}, \rho)}^{p}&=&\int_{\mathbb{R}^{N}}|f(u(x,s))|^{p}\rho(x)dx
\\
&\leq&a^p\int_{\mathbb{R}^{N}}\rho(x)dx
\\
&=&a^{p}.
\end{eqnarray*}
Thus
\begin{eqnarray*}
\|u(\cdot,t)\|_{L^{p}(\mathbb{R}^{N}, \rho)}&\leq& e^{-t}\|u(\cdot,0)\|_{L^{p}(\mathbb{R}^{N}, \rho)} +\int_{0}^{t}e^{s-t}\left[aK^{1/p}\|J\|_{L^{1}}+h \right]ds
\\
&=&e^{-t}\|u(\cdot,0)\|_{L^{p}(\mathbb{R}^{N}, \rho)} +R.
\end{eqnarray*}

Therefore, for any $t> \ln\left(\frac{\|u(\cdot,0)\|_{L^{p}(\mathbb{R}^{N}, \rho)}}{\varepsilon}\right)$, we have $\|u(\cdot,t)\|_{L^{p}(\mathbb{R}^{N}, \rho)}<\varepsilon +R$, and the proof is complete.
\end{proof}

\begin{lemma}
Suppose that the hypotheses (H1)-(H4) hold. Then, for any $\eta>0$, there exists $t_{\eta}$ such that $S(t_{\eta})\mathcal{B}(0,R+\varepsilon)$ has a finite covering by balls of $L^{p}(\mathbb{R}^N, \rho)$ with radius smaller than $\eta$.\label{lema3.2}
\end{lemma}
\begin{proof} From Lemma \ref{lema3.1}, it follows that $\mathcal{B}(0,R+\varepsilon)$ is invariant. Now, the solution of (\ref{1.1}) with initial condition $u_{0}\in \mathcal{B}(0,R+\varepsilon)$ is given, by the variation of constant formula, by
$$
u(x,t)=e^{-t}u_{0}+\int_{0}^{t}e^{-(t-s)}[(J*(f \circ u))(x,s)+h]ds.
$$
Write
$$
v(x,t)=e^{-t}u_{0}(x) \,\, \mbox{and} \,\,
w(x,t)=\int_{0}^{t}e^{-(t-s)}[(J*(f \circ u))(x,s)+h]ds.
$$
Let $\eta >0$ given. We may find $t(\eta)$ such that if $t \geq
t(\eta)$, then $\|v(\cdot,t)\|_{L^{p}(\mathbb{R}^{N}, \rho)} \leq
\frac{\eta}{2}$.

Now, using (H3), we obtain
\begin{eqnarray*}
||J\ast(f\circ u)(\cdot,s)||^p_{L^p(\mathbb{R}^N,\rho)} &=& \int_{\mathbb{R}^N}|J\ast(f\circ u)(x,s)|^p\rho(x)dx
\\
&=& \int_{\mathbb{R}^N}\left|\int_{\mathbb{R}^N}J(x-y)f(u(y))dy\right|^p\rho(x)dx
\\
&\leq& \int_{\mathbb{R}^N}\left(\int_{\mathbb{R}^N}J(x-y)|f(u(y))|dy\right)^p\rho(x)dx
\\
&\leq& \int_{\mathbb{R}^N}\left(a\int_{\mathbb{R}^N}J(x-y)dy\right)^p\rho(x)dx
\\
&=& \int_{\mathbb{R}^N}(a||J||_{L^1})^p\rho(x) dx
\\
&=& (a||J||_{L^1})^p\int_{\mathbb{R}^N}\rho(x)dx
\\
&=&(a||J||_{L^1})^p.
\end{eqnarray*}
Thus,
\begin{equation*}
||J\ast(f\circ u)(\cdot,s)||_{L^p(\mathbb{R}^N,\rho)} \leq a||J||_{L^1}.
\end{equation*}

Hence
\begin{eqnarray}
\|w(\cdot,t)\|_{L^{p}(\mathbb{R}^{N}, \rho)}&\leq&
\int^t_0e^{-(t-s)}(\|J\ast(f\circ
u)(\cdot,s)\|_{L^p(\mathbb{R}^N,\rho)}+\|h\|_{L^p(\mathbb{R}^N,\rho)})ds\nonumber
\\
&\leq &\int^t_0e^{-(t-s)}(a\|J\|_{L^1}+h)ds  \nonumber
\\
&=& a\|J\|_{L^{1}}+h.  \label{6.8}
\end{eqnarray}

On the other hand, by (H3), we have
\begin{eqnarray} \label{eqv8.2}
|w(x,t)| &\leq& \int_0^te^{-(t-s)}[|J\ast(f\circ u)(x,s)| +h] ds \nonumber
\\
&=&\int_0^te^{-(t-s)}\left| \int_{\mathbb{R}^N}J(x-y)f(u(y,t))dy\right|+h\ ds \nonumber
\\
&\leq&\int_0^te^{-(t-s)}\left(\int_{\mathbb{R}^N}J(x-y)|f(u(y,t))|dy+h\right) ds \nonumber
\\
&\leq& \int_0^te^{-(t-s)}\left(a\int_{\mathbb{R}^N}J(x-y)dy+h\right) ds \nonumber
\\
&=& \int_0^te^{-(t-s)}\left(a||J||_{L^1}+h\right) ds \nonumber
\\
&=& a||J||_{L^1} + h.
\end{eqnarray}

Furthermore, differentiating with respect to $x_{i}$, for $t\geq 0$,
we have
$$
\frac{\partial w}{\partial
x_{i}}(x,t)=\int_{0}^{t}e^{-(t-s)}\partial_{x_{i}} J*(f\circ
u)(x,s)ds, \,\, i=1,\cdots , N.
$$
Thus
\begin{eqnarray*}
\left|\frac{\partial w }{\partial x_{i}}(x, t) \right|&\leq& \int_{0}^{t}e^{-(t-s)}|\partial_{x_{i}}J*(f\circ u)(x, s)|ds.
\end{eqnarray*}
But, using (H4), obtain
\begin{eqnarray*}
|\partial_{x_{i}}J*(f\circ u)(x, s)| &\leq& \int_{\mathbb{R}^N} a\left|\partial_{x_i}J(x-y)\right|dy
\\
&\leq& aS <\infty,
\end{eqnarray*}
it follows that
\begin{equation}
\left|\frac{\partial w }{\partial x_{i}}(x, t) \right|\leq \int_0^te^{-(t-s)}aS\ ds \leq aS <\infty. \label{6.9}
\end{equation}

Now, let $l>0$ be chosen such that
\begin{equation}
(a\|J\|_{L^{1}}+h)\left(\int_{\mathbb{R}^{N}}(1-\chi_{B[0,l]})^{p^2/(p-1)}(x)\rho(x)dx\right)^{(p-1)/p^{2}}
\leq \frac{\eta}{4},\label{6.11}
\end{equation}
where $\chi_{B[0,l]}$ denotes the characteristic function of the ball $B[0,l]$. Then, using (\ref{6.8}), (\ref{eqv8.2}) and (\ref{6.11}), we obtain

\begin{eqnarray*}
\|(1-\chi_{B[0,l]})(\cdot)w(\cdot,t)\|_{L^{p}(\mathbb{R}^{N}, \rho)}^{p}&=& \int_{\mathbb{R}^N}|(1-\chi_{B[0,l]})(x)w(x,t)|^p\rho(x)dx
\\
&=&\int_{\mathbb{R}^N}|(1-\chi_{B[0,l]})(x)|^p|w(x,t)|^p\rho(x)dx.
\end{eqnarray*}
Using (\ref{conjugado}) and Holder's inequality, follows that
\begin{eqnarray*}
\lefteqn{\|(1-\chi_{B[0,l]})(\cdot)w(\cdot,t)\|_{L^{p}(\mathbb{R}^{N}, \rho)}^{p}=}
\\
&=&\int_{\mathbb{R}^N}|w(x,t)|\rho(x)^{1/p}|(1-\chi_{B[0,l]})(x)|^p|w(x,t)|^{p-1}\rho(x)^{(p-1)/p}dx
\\
&\leq& \left(\int_{\mathbb{R}^N}|w(x,t)|^p\rho(x)dx\right)^{1/p}\left(\int_{\mathbb{R}^N}|(1-\chi_{B[0,l]}(x)|^{p^2/(p-1)}|w(x,t)|^p\rho(x)dx\right)^{(p-1)/p}
\\
&=&||w(\cdot,t)||_{L^p(\mathbb{R}^N,\rho)}\left(\int_{\mathbb{R}^N}|(1-\chi_{B[0,l]})(x)|^{p^2/(p-1)}|w(x,t)|^p\rho(x)dx\right)^{(p-1)/p}
\\
&\leq& (a||J||_{L^1}+h)\left(\int_{\mathbb{R}^N}|(1-\chi_{B[0,l]})(x)|^{p^2/(p-1)}(a||J||_{L^1}+h)^p\rho(x)dx\right)^{(p-1)/p}
\\
&=&
(a||J||_{L^1}+h)^{p}\left(\int_{\mathbb{R}^N}|(1-\chi_{B[0,l]})(x)|^{p^2/(p-1)}\rho(x)dx\right)^{(p-1)/p}
\\
&<& \frac{\eta}{4}.
\end{eqnarray*}

Also, by (\ref{eqv8.2}) and (\ref{6.9}), the restriction of
$w(\cdot,t)$ to the ball $B[0,l]$ is bounded in $W^{1,p}(B[0,l])$
(by a constant independent of $u_{0}\in
\mathcal{B}(0,R+\varepsilon)$ and of $t$), and therefore the set
$\{\chi_{B[0,l]}w(\cdot,t)\}$ with $w(\cdot,0)\in
\mathcal{B}(0,R+\varepsilon)$ is relatively compact subset of
$L^{p}(\mathbb{R}^{N}, \rho)$ for any $t>0$ and, hence, it can be
covered by a finite number of balls with radius smaller than
$\frac{\eta}{4}$.

Therefore, since
$$
u(\cdot,t)=v(\cdot,t)+\chi_{B[0,l]}w(\cdot,t)+(1-\chi_{B[0,l]})w(\cdot,t),
$$
it follows that $S(t_{\eta})\mathcal{B}(0,R+\varepsilon)$ has a finite covering by balls of $L^{p}(\mathbb{R}^{N},\rho)$ with radius smaller than $\eta$, 
and the result is proved.

\end{proof}

 \vspace{1mm}

 Denoting by $\omega(C)$ the $\omega$-limit of a set $C$, we obtain the result
 below, whose proof is omitted because it is very similar to  Theorem 3.3 in
 \cite{Silva2}.

\begin{theorem}
Assume the same hypotheses of  Lemma \ref{lema3.2}. Then $\mathcal{A}=\omega(\mathcal{B}(0,R+\varepsilon))$, is a global attractor for the flow $S(t)$ generated by (\ref{1.1}) in $L^{p}(\mathbb{R}^{N}, \rho)$ which is contained in the ball of radius $R$.\label{teorema3.3}
\end{theorem}



\section{Boundedness results}

In this section we prove uniform estimates for the attractor whose
existence was given in the Theorem \ref{teorema3.3}.

\begin{theorem} \label{thm6.5}
Assume the same hypotheses of  Lemma \ref{lema3.2}. Then the attractor $\mathcal{A}$ belongs to the ball $\|\cdot\|_{L^{\infty}(\mathbb{R}^{N})}\leq r$, where $r=a\|J\|_{L^{1}}+h$.
\end{theorem}

\begin{proof}
  Let $u(x,t)$ be a solution of \eqref{1.1} in $\mathcal{A}$. Then, as we see in \eqref{emL2a},
$$
u(x,t)=\int_{-\infty}^{t}e^{-(t-s)}[J*(f\circ u)(x,s)+h]ds,
$$
where the equality above is in the sense of $L^{p}(\mathbb{R}^{N}, \rho)$. Thus, using (H3), obtain
\begin{align*}
|u(x,t)|
&\leq  \int_{-\infty}^{t}e^{-(t-s)}[|J*(f\circ u)(x,s)|+h]ds\\
&\leq \int_{-\infty}^{t}(a\|J\|_{L^{1}}+h)e^{-(t-s)}ds\\
&= \int_{-\infty}^{t}r e^{-(t-s)}ds = r.
\end{align*}
\end{proof}

%
%
%

Proceeding as in \cite{Silva2}, replacing
$\|\cdot\|_{L^{2}(\mathbb{R}, \rho)}$ by
$\|\cdot\|_{L^{p}(\mathbb{R}^{N}, \rho)}$, we obtain the following
result.

\begin{theorem} \label{thm7.2}
Assume the same hypotheses as in Lemma \ref{lema3.2}. Then, fixed $J_{0}$, for $J$ close to $J_{0}$ , the family of attractors $\{\mathcal{A}_{J}\}$ satisfies:
$$
\cup_{J}\mathcal{A}_{J}\subset \mathcal{B}[0,R],
$$
and furthermore, it is upper semicontinuous with respect to $J$ at $J_{0}$, that is
$$
\sup_{x\in \mathcal{A}_{J}} \inf_{y\in \mathcal{A}_{J_{0}}}\|x-y\|_{L^{p}(\mathbb{R}^{N}, \rho)}\to 0, \quad\text{as }J\to J_{0}.
$$
\end{theorem}

\section{Existence of energy functional} \label{functional}

In this section, we exhibit a energy functional for the flow of (\ref{1.1}), which decreases along of solutions (\ref{1.1}). For this, beyond hypotheses (H1)-(H4), we assume the following additional hypothesis on $f$:

\begin{itemize}
\item[(H5)] the nondecreassing function $f$ takes values between $0$ and $a$ and satisfying, for $0\leq s\leq a$
\begin{equation}
\left|\int_{0}^{s}f^{-1}(r)dr\right|< L<\infty.\label{H5}
\end{equation}

\item[ (H6) ] $f$ satisfies
$$
\int_{\mathbb{R}^{N}}|f(u(x))-f(u_{0})|dx<\infty.
$$
\end{itemize}

\begin{remark}
The hypothesis (H6) always occurs, for example, in fields with finite excited region, which tend to resting state, when $|x|\to \infty.$
\end{remark}

Motivated by energy functionals from \cite{French}, \cite{Giese}, \cite{Kubota} and \cite{Silva4}, we define 
$F:L^{p}(\mathbb{R}^{N}, \rho) \rightarrow \mathbb{R}$  by
\begin{equation}
F(u)=\int_{\mathbb{R}^{N}}\left[-\frac{1}{2}f(u(x))\int_{\mathbb{R}^{N}}J(x-y)f(u(y))dy+\int_{0}^{f(u(x))}f^{-1}(r)dr
-h f(u(x))\right]dx.\label{L1}
\end{equation}

\begin{remark}
The similar functional given in \cite{Silva4} is well defined in
whole phase space. Unfortunately this does not occur here, because
the functional given in (\ref{L1}) can take values $\pm \infty$. An
example where this occurs is when whole field is at homogeneous
resting state with constant membrane potential $u_{0}$. In this
case, the external stimulus applied, $h$, satisfies
$h=u_{0}-\|J\|_{L^{1}}f(u_{0})$. \label{fields homogeneous}
\end{remark}

Let $u_{0}$ be a equilibrium solution for (\ref{1.1}), which it is
given implicitly by equation
$$
u_{0}=\|J\|_{L^{1}}f(u_{0})+h.
$$
Write $U=u-u_{0}$ and $g(U)=f(U+u_{0})-f(u_{0})$. Then the equation
(\ref{1.1}) can be write as
\begin{equation}
\frac{\partial U}{\partial t}(x,t)=-U(x,t)+ J*(g\circ U)(x,t). \label{hat}
\end{equation}
For equation (\ref{hat}), we define the functional
\begin{equation}
G(U)=\int_{\mathbb{R}^{N}}\left[-\frac{1}{2}g(U(x))\int_{\mathbb{R}^{N}}J(x-y)g(U(y))dy
+\int_{0}^{g(U(x))}g^{-1}(r)dr\right]dx.\label{Lhat}
\end{equation}
Thus we obtain the following result:

\begin{theorem} \label{LyapunovBouded} Let $U(\cdot,t)$ be
a solution of (\ref{hat}). Then, under the hypotheses, (H3), (H5)
and (H6), we have
\begin{multline}
G(U)=\int_{\mathbb{R}^{N}}\left[-\frac{1}{2}[f(u(x))-f(u_{0})]\int_{\mathbb{R}^{N}}J(x-y)
[f(u(y))-f(u_{0})]dy \right.
\\ \left. +\int_{f(u_{0})}^{f(u(x))}f^{-1}(r)dr\right]dx
<\infty
\end{multline}
and
\begin{equation}
 \frac{d
}{dt}G(U(x,t))=-\int_{\mathbb{R}^{N}}f'(u(x,t))\left(\frac{\partial
u}{\partial t}(x,t)\right)^{2}dx\leq 0.
\end{equation}
\end{theorem}

\begin{proof}
Since $g(U)=f(U+u_{0})-f(u_{0})$, from equation (\ref{Lhat}), we
obtain
\begin{multline*}
G(U)=\int_{\mathbb{R}^{N}}\left[-\frac{1}{2}[f(U(x)+u_{0})-f(u_{0})]
\int_{\mathbb{R}^{N}}J(x-y) [f(U(y)+u_{0})-f(u_{0})]dy \right.
\\
\left. +\int_{0}^{g(U(x))}g^{-1}(r)dr\right]dx.
\end{multline*}
Now, using that $U=u-u_{0}$, $g(0)=0$ and the fact that $f^{-1}$ and
$g^{-1}$ differ only by translation, which is an isometry, follows
that
\begin{eqnarray*}
\int_{f(u_{0})}^{f(u(x))}f^{-1}(r)dr=\int_{0}^{f(U(x)+u_{0})-f(u_{0})}f^{-1}(r)dr=\int_{0}^{g(U(x))}g^{-1}(r)dr.
\end{eqnarray*}
Hence
\begin{multline*}
G(U)=\int_{\mathbb{R}^{N}}\left[-\frac{1}{2}[f(u(x))-f(u_{0})]
\int_{\mathbb{R}^{N}}J(x-y) [f(u(y))-f(u_{0})]dy\right.
\\
\left. +\int_{f(u_{0})}^{f(u(x))}f^{-1}(r)dr\right]dx.
\end{multline*}
From hypotheses (H5) and (H6), it follows that $|G(U)|<\infty$.

Furthermore, proceeding as in the Theorem 4.4 of \cite{Silva4}, it
is easy to verify that
\begin{equation*}
\frac{d
}{dt}G(U(x,t))=-\int_{\mathbb{R}^{N}}g'(U(x,t))\left(\frac{\partial
U}{\partial t}(x,t)\right)^{2}dx.
 \label{Lyapunovhat}
\end{equation*}
Hence
\begin{eqnarray*}
 \frac{d }{dt}G(U(x,t))&=&-\int_{\mathbb{R}^{N}}g'(U(x,t))\left(\frac{\partial
U}{\partial t}(x,t)\right)^{2}dx\\
&=&-\int_{\mathbb{R}^{N}}[f'(U(x,t)+u_{0})-\frac{d}{dt}(f(u_{0}))]\left(\frac{\partial
u}{\partial t}(x,t) - \frac{\partial u_{0}}{\partial t}\right)^{2}dx\\
&=&-\int_{\mathbb{R}^{N}}[f'(u(x,t))]\left(\frac{\partial
u}{\partial t}(x,t) \right)^{2}dx.
\end{eqnarray*}
From hypothesis (H3) the result follows.
\end{proof}

\begin{remark}
From Theorem \ref{LyapunovBouded}, it follows that the functional
given in (\ref{Lhat}) is actually a Lyapunov functional for the flow
generated by equation (\ref{hat}).
\end{remark}

\section{Concluding Remarks}

In this paper we extend results on global dynamical of the neural
fields equation considering fields in $x\in \mathbb{R}^{N}$ and more
abstracts phase spaces. Although realistically, $N$ should be equal
to $1,2$ or $3$, we do not restrict the calculations to this case,
because all estimates also are valid with $N>3$. Furthermore,
motivated by energy functional existing in the literature, we
exhibit one functional energy (type Lyapunov functional), which is
well defined throughout phase space, which it is a lot important for
studying existence and stability of solutions of equilibria of
equations neural fields.

\subsection*{Acknowledgments}
The authors would like to thank the professors Antonio L. Pereira
(USP), and Flank D. M. Bezerra (UFPB) for their suggestions for this
work.

\end{document}